\documentclass[12pt,a4paper, reqno]{amsart}
\usepackage{amsmath,amsfonts,amssymb,amsthm,amscd}
\usepackage{graphicx}

\input{comment.sty}
\includecomment{MM}
\includecomment{CL}

\renewcommand{\b}{\beta}
\newcommand{\g}{\gamma}

\renewcommand{\d}{\delta}
\newcommand{\D}{\Delta}

\renewcommand{\l}{\lambda}

\newcommand{\m}{\mu}

\renewcommand{\o}{\omega}

\newcommand{\s}{\sigma}

\renewcommand{\th}{\theta}
\newcommand{\e}{\varepsilon}

\newcommand{\E}{{\mathbb E}}

\newcommand{\R}{{\mathbb R}}

\newcommand{\Pb}{{\mathbb P}}

\newtheorem{theorem}{Theorem}[section]

\newtheorem{lemma}[theorem]{Lemma}
\newtheorem{corollary}[theorem]{Corollary}
\newtheorem{conjecture}[theorem]{Conjecture}

\newtheorem{question}[theorem]{Question}

\theoremstyle{definition}

\newtheorem{notation}[theorem]{Notation}

\theoremstyle{remark}
\newtheorem{remark}[theorem]{Remark}

\begin{document}

\title{Random Dirichlet series arising from records}
\author{Ron Peled, Yuval Peres, Jim Pitman, and Ryokichi Tanaka}

\address{Tel Aviv University, School of Mathematical Sciences, Tel Aviv, 69978, Israel}
\email{peledron@post.tau.ac.il}

\address{Microsoft Research, 1 Microsoft Way, Redmond, WA, 98052, USA}
\email{peres@microsoft.com}

\address{Department of Statistics, University of California Berkeley, Berkeley, CA, 94720, USA}
\email{pitman@stat.berkeley.edu}

\address{Tohoku University, 2-1-1 Katahira, Aoba-ku, Sendai, 980-8577, Japan}
\email{rtanaka@m.tohoku.ac.jp}

\date{\today.\\
Research of R.P.\ is supported by an ISF grant and an IRG grant.\\
Research of R.T.\ is supported by JSPS Grant-in-Aid for Young Scientists (B) 26800029.
}
\begin{abstract}
We study the distributions of the random Dirichlet series with parameters $(s, \b)$ defined by
$$
S=\sum_{n=1}^{\infty}\frac{I_n}{n^s},
$$
where $(I_n)$ is a sequence of independent Bernoulli random variables, $I_n$ taking value $1$ with probability $1/n^\b$ and value $0$ otherwise.
Random series of this type are motivated by the record indicator sequences which have been studied in extreme value theory in statistics.
We show that when $s>0$ and $0< \b \le 1$ with $s+\b>1$ the distribution of $S$ has a density; otherwise it is purely atomic or not defined because of divergence.
In particular, in the case when $s>0$ and $\b=1$, we prove that for every $0<s<1$ the density is bounded and continuous, whereas for every $s>1$ it is unbounded.
In the case when $s>0$ and $0<\b<1$ with $s+\b>1$, the density is smooth.
To show the absolute continuity,
we obtain estimates of the Fourier transforms,
employing van der Corput's method to deal with number-theoretic problems.
We also give further regularity results of the densities,
and present an example of non-atomic singular distribution which is induced by the series restricted to the primes.
\end{abstract}
\maketitle

\section{Introduction}

The purpose of this paper is to analyse a class of probability distributions defined by infinite series of the following type:
Let $I_1, I_2, \dots$ be a sequence of independent random variables $I_n$ taking values $0$ or $1$ with $\Pb\left(I_n=1\right)=1/n$.
Define a random series by
\begin{equation}\label{original}
S_{o}:=\sum_{n=1}^{\infty}\frac{I_n}{n}.
\end{equation}
Note that the series converges almost surely since its expectation is finite, and its distribution has the support $[1, \infty)$.
A central question we consider is whether this distribution has a density or not.
We show this distribution does have a density, but leave open the questions of whether this density is bounded, or continuous.

This distribution arises from the study of records in statistics.
Let $U_1, U_2, \ldots $ be a sequence of independent uniform $[0,1]$ variables, and let
$I_1, I_2, \ldots $ be the associated sequence of {\em record indicators}:
\begin{equation}\label{recinds}
I_n:= 1 ( U_n > U_j \mbox{ for all } 1 \le j < n )
\end{equation}
meaning that $I_n = 1$ if $U_n$ exceeds all previous values, and $I_n = 0$ otherwise.
R\'enyi \cite{R} showed that the record indicators are independent with $\Pb(I_n = 1) = 1/n$ for
all $n \ge 1$. Related properties of the record indicator sequence and its partial sums, counting numbers of
records, have been extensively studied. See e.g.\ the monographs of Arnold et al.\ \cite{ABN}, 
and Nevzorov \cite{N}.
See also \cite[Chapter 3]{P} for related topics and references therein.
We show in Theorem \ref{condexp}
that the conditional expectation of $U_1$ given $I_1, I_2, \dots$ is
\begin{equation}\label{eq2}
\E[U_1 | I_1, I_2, \dots]
= \prod_{n=2}^{\infty}\left(1-\frac{I_n}{n}\right).
\end{equation}
In this setting, the random series (\ref{original}) approximates the logarithm of (\ref{eq2}).

\begin{figure}[h]
	\begin{center}
	\includegraphics[width=145mm]{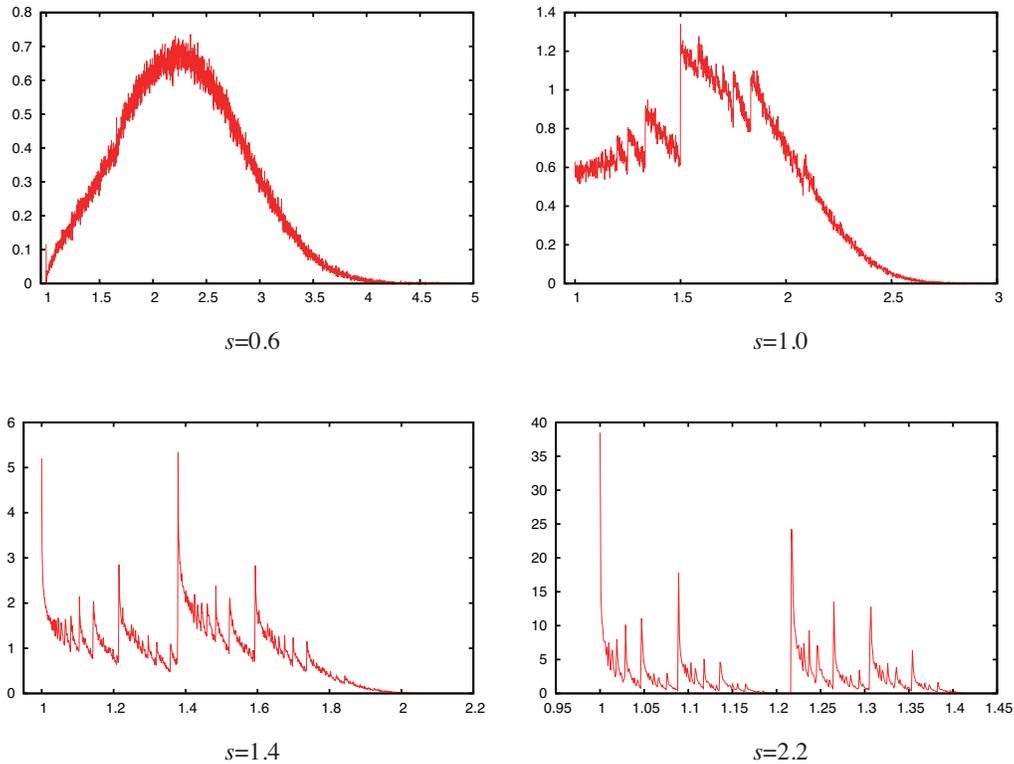}
	\end{center}
	\caption{Simulation approximations to the densities of $
S:=\sum_{n=1}^{N}I_n/n^s$ ($I_n=1$ with probability $1/n$ and $0$ otherwise) for $s=0.6, 1.0, 1.4$ and $2.2$, where $N=10^4$, the number of samples is $10^6$ and the bin size is $10^{-3}$.}
	\label{graphs}
\end{figure}

From the viewpoint of the study of random series,
it is natural to parameterize the series (\ref{original}) as follows:
For $s>0$, let
\begin{equation}\label{var}
S:=\sum_{n=1}^{\infty}\frac{I_n}{n^s},
\end{equation}
where $I_1, I_2, \dots$ is a sequence of independent random variables $I_n$ with values $0$ or $1$ with probability $1-1/n^\b$ or $1/n^\b$, respectively, with $\b$ a positive parameter.  The series converges almost surely if $s+\b > 1$, and just diverges almost surely if $s+\b \le 1$ by Kolmogorov's three-series theorem \cite[Theorem 2.5.4]{D}.
We recover (\ref{original}) when $s=1$ and $\b=1$.
Let $\m$ be the distribution of $S$ defined by (\ref{var}).
Jessen and Wintner showed that every convergent infinite convolution of discrete measures is of {\it pure type}: 
it is either atomic singular (purely discontinuous), non-atomic singular (continuous singular), or absolutely continuous with respect to Lebesgue measure \cite{JW}.
We observe that if $\b > 1$, then the sequence $(I_n)$ consists of only finitely many ones almost surely, hence $\m$ is atomic singular, and in fact, it is supported on the countable set of all possible values of the finite sums $\sum \e_n/n^s$, where $\e_n$ is $0$ or $1$.

\begin{theorem}\label{main}
Let $s>0$ and $0< \b \le 1$ with $s + \b > 1$.
	\begin{itemize}
	\item[(1)] For $\b=1$,
the distribution $\m$ is absolutely continuous with respect to Lebesgue measure for all $s > 0$.
	For every $0 < s < 1$, the distribution $\m$ has a bounded continuous density, whereas for every $s>1$, it has an unbounded density.
	Moreover, for each $s >0$ the Fourier transform $\hat \m$ of $\m$ has the following property: for every small enough $\e>0$ there exists a constant $C_{\e,s} > 0$ such that for all real $t$,
	\begin{equation}\label{eqthm1}
		|\hat \m(t)| \le C_{\e,s} |t|^{-\frac{1}{s}+\e}.
	\end{equation}
	\item[(2)]
For $0 < \b < 1$, for every $s> 1 - \b$,
	the distribution $\m$ has a smooth density.
	Moreover,
	there exist constants $C_{\b,s}>0$ and $T>0$ such that for all real $t$ with $|t| \ge T$,
	\begin{equation*}
		|\hat \m(t)| \le \exp\left(-C_{\b,s} |t|^{\frac{1-\b}{s+1}}\right).
	\end{equation*}
	\end{itemize}
\end{theorem}

In Section \ref{regularity} and Section \ref{ac}, we prove this theorem;
see Theorem \ref{fourier}, Theorem \ref{all-s} and Theorem \ref{unbounded}.
In part 1 of Theorem \ref{main}, we observe a transition of boundedness of the densities on the line $\b=1$ at $s=1$.
This leaves open questions about the density in the critical case when $\b=1$ and $s=1$, which we present in Section \ref{questions}.
Figure \ref{graphs} displays the densities of $\m$ for $\b=1$ with several values of $s$.

Random series have been studied in various contexts.
If one considers the random harmonic series $\sum_{n=1}^{\infty}\pm 1/n$, where the signs are chosen independently with equal probability,
then one can show that the distribution has a smooth density \cite{Sc}.
In the context of random functions, the random series $\sum_{n=1}^{\infty}\pm 1/n^s$ is called a {\em random Dirichlet series} and has been studied mainly by focusing on its analytic properties (e.g., \cite{BM} and references therein).
In most cases, however, it is assumed that the random signs or the random coefficients are independent identically distributed or satisfy certain uniformity.
The random Dirichlet series with coefficients $I_n$ as in (\ref{var}) has not been extensively studied yet.
The random Dirichlet series (\ref{var}) can also be compared with the random geometric series $\sum_{n=1}^{\infty}\pm \l^n$, where the signs are chosen independently with equal probability and $\l$ is a parameter between $0$ and $1$.
This distribution has been studied under the name of Bernoulli convolutions,
and its absolute continuity/singularity problem has attracted a lot of attention.
See the expository article by Peres, Schlag and Solomyak \cite{PSS} and recent notable progress by Hochman \cite{H} and by Shmerkin \cite{S}.

Let us show further regularity results of the densities of $\m$.
The following is an implication of the decay of the Fourier transform in the case when $0<s<2$ and $\b=1$.

\begin{corollary}\label{HausYoung}
	Let $s>0$ and $\b=1$.
	For an integer $r \ge 0$ and for every $0 < s < 1/(r+1)$, the distribution $\m$ has a density in $C^r$, and
	for every $1 \le s < 2$, it has a density in $L^q$ for every $1 \le q < s/(s-1)$, where $s/(s-1)=\infty$ when $s=1$.
\end{corollary}

Let us mention another implication of the decay of the Fourier transform for $0 < s < 2$ and $\b=1$.
The fractional derivatives are expressed in terms of the $(2, \g)$-Sobolev space $L^2_\g$,
where the norm is defined by
$\|\m\|_{2,\g}^2=\int_{-\infty}^{\infty}|\hat \m(t)|^2|t|^{2\g}dt$.
Finiteness of $\|\m\|_{2,\g}$ for some positive $\g > 0$, implies that $\m$ has $\g$-fractional derivatives in $L^2$.
The following is an immediate consequence of part 1 of Theorem \ref{main}.

\begin{corollary}\label{fracderiv}
Suppose that $0 < s < 2$ and $\b=1$.
For every $0 < \g < \frac{1}{s}-\frac{1}{2}$, the density of $\m$ has $\g$-fractional derivatives in $L^2$.
\end{corollary}

In the proof of Theorem \ref{main}, we require an estimate of exponential sums by Weyl and van der Corput (\cite{GK}, \cite{KN})
to bound the Fourier transform of the distribution $\m$.
In the cases when $0<s<2$ and $\b=1$, and when $s>0$ and $0<\b<1$ with $s+\b>1$, respectively,
the decay of the Fourier transform suffices to conclude that the distribution $\m$ has a density,
since $\hat \m$ is in $L^2$ in these parameter sets,
and further regularity results also follow from these estimates.
In the case when $s \ge 2$ and $\b=1$, however, the absolute continuity of the distribution $\m$ does not follow directly from the estimate of the Fourier transform (\ref{eqthm1}) in part 1 of Theorem \ref{main}.
In Section \ref{ac}, we show that $\m$ is absolutely continuous for all $s>0$ and $\b=1$, employing a conditioning argument combined with van der Corput's method.
We remark that the method which we use there does not yield further regularity of the densities for $s\ge2$ and $\b=1$ unlike the case for $0<s<2$ and $\b=1$.
For the sharpness of the estimate of the Fourier transform for $0<s<2$ and $\b=1$, see Remark \ref{sharp}.


To summarize the results, for $(s, \b)$ with $s+\b >1$, the distribution $\m$ of (\ref{var}) is always absolutely continuous except
for the trivial case when $\beta >1$.

We present one more result which provides a non-atomic singular distribution, restricting to the prime numbers sequence:
Let $p$ be a prime.
Consider an independent sequence $(I_p)$ with value $0$ or $1$ with probability $1-1/p$ or $1/p$, respectively, and the following random series:
$$
S_{primes}:=\sum\frac{I_p}{p^s},
$$
where the summation runs over all primes $p$ and $s>0$.
Notice again that $S_{primes}$ is finite almost surely.

\begin{theorem}\label{sing}
For every $s>0$, the distribution $\m$ of $S_{primes}$ is non-atomic singular.
\end{theorem}

Erd\H{o}s proved that the asymptotic distribution of the additive function $f_n=\sum 1/p^s$, where the summation runs over all prime divisors of $n$ is non-atomic singular \cite{E}.
The proof of Theorem \ref{sing} uses essentially the same method as Erd\H{o}s'.

The organisation of this paper is as follows:
In Section \ref{records}, we discuss the record indicator and the background, and then obtain the formula for the conditional expectation (\ref{eq2}) in Theorem \ref{condexp}.
In Section \ref{regularity}, we show a part of Theorem \ref{main} concerning the estimate of the Fourier transform in Theorem \ref{fourier} and deduce Corollary \ref{HausYoung}.
In Section \ref{ac}, we prove the absolute continuity of the distribution $\m$ for all $s>0$ and $\b=1$ in Theorem \ref{all-s}, and the unboundedness of the densities for all $s>1$ and $\b=1$ in Theorem \ref{unbounded}.
Then we complete the proof of Theorem \ref{main}.
In Section \ref{primes}, we prove Theorem \ref{sing}.
In Section \ref{questions}, we discuss the boundedness of the density at the critical case when $s=1$ and $\b=1$, and present some open problems.

\begin{notation}
Throughout this article, we use $c, C, c_1, c_2, \dots$, to denote absolute constants whose exact values may change from line to line, and also use them with subscripts, for instance, $C_\e$ to specify its dependence only on $\e$.
For functions $f$ and $g$, we write $f \asymp g$ if there exist some absolute constants $c_1, c_2 >0$ such that $c_1 g \le f \le c_2 g$.
\end{notation}

\section{Records and probabilistic motivations}\label{records}

We start with the following question about the sequence of record indicators
$I_1, I_2, \ldots $ derived from independent uniform $[0,1]$ variables $U_1, U_2, \ldots$
as in \eqref{recinds}.
How much information does the sequence of record indicators reveal about the value of $U_1$?
The answer, provided below, may be compared with the answer to the corresponding question if
$(I_1,I_2, \ldots )$ is replaced by $(X_1,X_2, \ldots )$ where $X_n:= 1(U_n < U_1)$.
In this case, $U_1$ is recovered from $(X_1, X_2, \ldots )$ with probability one as the almost sure limit of $S_n/n$ as $n \to \infty$, where $S_n:= \sum_{k=1}^n X_k$ is the number of ones in the first $n$ places.
For the record indicators $(I_1,I_2, \ldots )$, the value $U_1$ cannot be fully recovered from the record indicators $(I_1,I_2, \ldots )$.
More precisely, we establish the following theorem:

\begin{theorem}\label{condexp}
Let $V_\infty:=\E[U_1 | I_1, I_2, \dots]$. Then
\begin{equation*}
V_\infty=\prod_{n=2}^{\infty}\left(1-\frac{I_n}{n}\right).
\end{equation*}
Moreover,
\begin{equation}\label{meansq-value}
\E V_{\infty}^2=\frac{-\cos \left(\sqrt{5}\pi/2\right)}{\pi}=0.29667513474359 \cdots < \frac{1}{3}=\E\left(U_1^2\right).
\end{equation}
\end{theorem}

Since
$$
\E( U_1 - V_\infty)^2 = \E(U_1^2) - \E(V_\infty^2) >0
$$
the difference between $\E V_\infty^2$ and $\E U_1^2$ reflects the fact that $U_1$ is not a measurable function
of all the record indicators $(I_1, I_2, \dots)$.
The Jessen-Wintner law of pure types implies that the distribution of the infinite product $V_\infty$ is either singular, or absolutely continuous with respect to Lebesgue measure \cite{JW}.
Define a random series by taking the logarithm and the positive sign:
$$
S_{\infty}:=-\log V_\infty =-\sum_{n=2}^{\infty}\log \left(1-\frac{I_n}{n}\right).
$$
Since $-\log(1-I_n/n)\asymp I_n/n$, one can expect that the above sum is approximated by $S_o$ in (\ref{original}).
Actually, the following holds in the same way as in the case of $S_o$.

\begin{theorem}\label{Sinfty}
Let $\m_\infty$ be the distribution of $S_\infty$.
Then $\m_\infty$ has a density in $L^q$ for every $1\le q <\infty$.
In particular, the distribution of $V_\infty=\E[U_1 | I_1, I_2, \dots]$ is absolutely continuous with respect to Lebesgue measure.
\end{theorem}

We prove Theorem \ref{Sinfty} at the end of Section \ref{regularity} as a consequence of more general facts.
In this section we establish Theorem \ref{condexp}.
First, we obtain the conditional distribution of $U_1$ given the record indicators $(I_1,I_2, \ldots )$.
We denote by $beta(a, b)$ for $a, b>0$ the probability distribution on $[0,1]$ whose density at $u \in (0,1)$ relative to length measure is proportional to $u^{a-1}(1-u)^{b-1}$.

\begin{theorem}
\label{thm1}
For $n \ge 1$, let
$$
M_n:= \max_{1\le i \le n} U_i.
$$
Then
\begin{align}
U_1 = M_n \prod_{j=2}^n \frac{M_{j-1}}{M_{j}}
\label{ratios}
\end{align}
where the random variables $M_n$ and $M_{j-1}/M_{j}, 2 \le j \le n$ are independent, with
$M_n$ distributed beta$(n,1)$ and $M_{j-1}/M_{j}$ distributed as a mixture with weights $1 - 1/j$ and $1/j$ of
a point mass at $1$ and a beta$(j-1,1)$ distribution on $(0,1)$.

The conditional distribution of $U_1$ given $I_1, \ldots, I_n$ is described by \eqref{ratios} where
given $I_1, \ldots, I_n$ the $M_n$ and $M_{j-1}/M_{j}$ for $2 \le j \le n$ are conditionally independent, with
\begin{itemize}
\item[(i)] $M_{j-1}/M_{j} = 1$ if and only if $I_j=0$,
\item[(ii)] $M_n$ distributed beta$(n,1)$,
\item[(iii)] $M_{j-1}/M_{j}$ distributed beta $(j-1,1)$ if $I_j = 1$.
\end{itemize}
\end{theorem}
\begin{proof}
Note that $M_n$ is distributed $beta(n,1)$ since $\Pb(M_n < x)=\Pb(U_1 < x)^n$, and
the ratio $M_{j-1}/M_j$ is distributed as indicated since
$\Pb(M_{j-1}/M_j=1)=\Pb(I_j=0)=1-1/j$,
and for $x<1$, $\Pb(M_{j-1}/M_j < x)=\Pb(M_{j-1} < x U_j)=\int_0^1(xu)^{j-1}du$.
The asserted joint distribution of the $n$ factors in \eqref{ratios} is established by induction on $n$, using $M_{n+1} = \max(M_n, U_{n+1})$, where $U_{n+1}$ is independent of
$M_n$ and $M_{j-1}/M_{j}, 2 \le j \le n$.
Indeed, it is enough to show that $M_n$ and $M_{n-1}/M_n$ are independent and this can be checked directly.

It is clear by definition of
$I_j$ that (i) above holds. So the $I_j$ are functions of the independent ratios $M_{j-1}/M_{j}$, hence independent as $j$ varies with $\Pb(I_j = 1) = 1/j$,
as found by R\'enyi \cite{R}.
The independence of $M_n$ and the $I_j$ for $1 \le j \le n$ is well-known \cite[Lemma 13.2]{N}.
Thus, given $I_1,\dots, I_n$, the $M_n$ and the ratios $M_{j-1}/M_j$ for $2 \le j \le n$ are conditionally independent.
It follows easily that the conditional distribution of $M_n$ and $M_{j-1}/M_{j}$ given $I_1, \dots, I_n$ is as indicated in (ii) and (iii).
\end{proof}

Now we prove Theorem \ref{condexp}:

\begin{proof}[Proof of Theorem \ref{condexp}]

Since the mean of beta$(a,b)$ is $a/(a+b)$, we read from Theorem \ref{thm1} that
\begin{align}
V_n:=\E(U_1 | I_1, \ldots, I_n ) = \frac{n}{n+1} \prod_{j=2}^n \left( 1 - \frac{I_j}{j} \right).
\label{mg}
\end{align}
By the bounded martingale convergence theorem,
\begin{align}
V_\infty:=\E(U_1 | I_1, I_2, \ldots)  = \prod_{j=2}^\infty \left( 1 - \frac{  I_j } {j} \right).
\end{align}
Note in passing that the limiting infinite products considered here exist not only
almost surely, as guaranteed by martingale convergence, but in fact for all sequences of $0/1$ values of $I_2, I_3, \ldots$, allowing $0$ as a possible limit.  This is obvious by inspection of the infinite products, since the partial products are non-increasing.

The mean square of $V_\infty$ is given by
\begin{equation}\label{mean-2}
\E( V_\infty ^2 ) = \prod_{j=2}^\infty \left( 1 - \frac{1}{j} + \frac{1}{j} \left( 1 - \frac{1}{j} \right)^2 \right).
\end{equation}
 The $j$-th factor in \eqref{mean-2} is
\begin{align}
1 - \frac{1}{j} + \frac{1}{j} \left( 1 - \frac{1}{j} \right)^2 &=    \frac{ (j-1) ( j^2 + j - 1) }{j^3 } \\
&=    \frac{ (j-1) }{j^4} (  (j-1) (j+1)^2 + 1 ) \\
&=    \left( \frac{j-1}{j} \right) ^2 \left( \frac{j+1}{j} \right)^2 \left(  1 + \frac{1} {(j-1) (j+1)^2 } \right)
\end{align}
There is some telescoping of the product, with the simplification
$$
\E( V_\infty^2) = \frac{1}{4} \prod_{j=2}^\infty \left(  1 + \frac{1} {(j-1) (j+1)^2 } \right)
$$
with the finite $n$ version, using (\ref{mg}),
\begin{align}
\E( V_n^2) = \frac{1}{4} \prod_{j=2}^n \left(  1 + \frac{1} {(j-1) (j+1)^2 } \right),
\label{finprod}
\end{align}
which increases to its limit as $n$ increases.

To prove the formula \eqref{meansq-value}, we show first
that
the finite product in \eqref{finprod} can be evaluated as
\begin{equation}
\label{gamexp}
\E( V_n^2) =  \frac{ 5 n \Gamma \left( n + \frac{3}{2} - \frac{\sqrt{5}}{2} \right) \Gamma \left( n + \frac{3}{2} + \frac{\sqrt{5}}{2} \right) }
{ (n+1)!^2 \Gamma \left( \frac{7}{2} - \frac{\sqrt{5}}{2} \right) \Gamma \left( \frac{7}{2} + \frac{\sqrt{5}}{2} \right) }.
\end{equation}
Indeed,  this formula holds for $n=1$, with both sides equal to $1/4$, by interpreting the empty product in \eqref{finprod} as $1$, and using
the gamma recursion $\Gamma(r+1) = r \Gamma(r)$. The proof for general $n$ is by induction. Assuming that \eqref{gamexp} has been established for
$n$, the formula with $n+1$ instead of $n$ is deduced from the identity
$$
1+\frac{1}{n(n+2)^2}=
\frac{ (n+1)(n^2+3n+1) }
{n(n+2)^2},
$$
by using the gamma recursion to expand
$$
n^2+3n+1=
\left(  n + \frac{3}{2} \right)^2 - \left( \frac{\sqrt{5}}{2}\right) ^2 =
\frac{\Gamma \left(n+1 +\frac{3}{2}-\frac{\sqrt{5}}{2}\right)\Gamma \left(n+1 +\frac{3}{2}+\frac{\sqrt{5}}{2}\right)}{\Gamma \left(n+\frac{3}{2}-\frac{\sqrt{5}}{2}\right)\Gamma \left(n+\frac{3}{2}+\frac{\sqrt{5}}{2}\right)}.
$$
Euler's reflection formula for the gamma function
$$
\Gamma(1 - z ) \Gamma(z ) = \frac{ \pi } { \sin ( \pi z ) }
$$
applied to $z = 1/2 + x$ becomes
$$
\Gamma(1/2 - x ) \Gamma(1/2 + x  ) = \frac{ \pi } { \cos ( \pi x ) }.
$$
By repeated applications of $\Gamma(r+1) = r \Gamma(r)$ this yields
$$
\Gamma(7/2 - x ) \Gamma(7/2 + x  ) = ((1/2)^2 - x^2)( (3/2)^2 - x^2 ) ( (5/2)^2 - x^2 ) \frac{ \pi } { \cos ( \pi x ) }
$$
which for $x = \sqrt{5}/2$ reduces to
$$
\Gamma\left( \frac{7}{2} - \frac{\sqrt{5}}{2}  \right) \Gamma\left( \frac{7}{2} + \frac{\sqrt{5}}{2}  \right) = \frac{ - 5 \pi } { \cos ( \pi \sqrt{5}/ 2 ) } .
$$
Substituting this expression in \eqref{gamexp} and evaluating the limit with Stirling's formula $\Gamma(n+r) \sim (n/e)^n n ^{r - 1/2} \sqrt{2 \pi }$
yields \eqref{meansq-value}.
\end{proof}

\section{Estimates of Fourier transforms}\label{regularity}

Recall that
the random Dirichlet series
$
S=\sum_{n=1}^{\infty}I_n/n^s
$
with parameters $s>0$ and $\b>0$ is defined by an independent sequence of Bernoulli random variables $I_n$ taking value $1$ with probability $1/n^\b$ and $0$ otherwise.
We assume that $s+\b>1$ for the almost sure convergence.
Let $\m$ be the distribution of $S$.
Here we start with an estimate of the Fourier transform of $\m$,
$$
\hat \m(t) := \int_0^\infty e^{- 2 \pi i t x } d \mu(x) \qquad (-\infty < t < \infty ).
$$

\begin{theorem}\label{fourier}
Let $s>0$ and $0< \b \le 1$ with $s + \b > 1$.
	\begin{itemize}
	\item[(1)] Let $s>0$ arbitrary and $\b=1$.
	Then for every small enough $\e>0$ there exists a constant $C_{\e,s} > 0$ such that for every $t$,
	\begin{equation*}
		|\hat \m(t)| \le C_{\e,s} |t|^{-\frac{1}{s}+\e}.
	\end{equation*}
	In particular, for $0<s<2$ the distribution $\m$ has a density in $L^2$, and for $0<s<1$ it has a bounded continuous density.
	\item[(2)] Let $s>0$ arbitrary and $0 < \b < 1$ with $s+\b>1$.
	Then there exist constants $C_{\b,s} > 0$ and $T>0$ such that for every $|t| \ge T$,
	\begin{equation*}
		|\hat \m(t)| \le \exp\left(-C_{\b,s} |t|^{\frac{1-\b}{s+1}}\right).
	\end{equation*}
	In particular, the distribution $\m$ has a smooth density.
	\end{itemize}
\end{theorem}

For the Fourier transform of $\m$, we have
\begin{align*}
|\hat \m(t)|^2	&=\int_{\R^2}e^{-2\pi i t (x-y)}d\m(x)d\m(y) \\
			&=\prod_{n=1}^{\infty}\left(1-\frac{2}{n^\b}+\frac{2}{n^{2\b}}+2\left(\frac{1}{n^\b}-\frac{1}{n^{2\b}}\right)\cos\frac{2 \pi t}{n^s}\right).
\end{align*}

Since $|\hat \m(t)|^2$ is even in $t$, it is enough to estimate for $t >0$.
To prove Theorem \ref{fourier} (1),
we will show that for every $\e >0$, there exists an interval $I_t=(a(t), b(t)]$ such that
the above product which is restricted to $I_t$
has the desired bound.
It is realised by taking $I_t$ as $[t^{\frac{1}{q+2+s}}, t^{\frac{1}{s}}]$, where $q$ is a large enough integer.
To prove Theorem \ref{fourier} (2),
we will find an interval $I_t$ such that the above product which is restricted to $I_t$ decays sub-exponentially fast.
The interval $I_t$ is chosen as $[t^{\frac{1}{s+2}}, t^{\frac{1}{s+1}}]$.

We begin with a lemma which involves an estimate of exponential sums.

\begin{lemma}\label{Iq-sigma}
Fix $s>0$.
For an integer $q \ge 0$,
suppose that $f$ has $q+2$ continuous derivatives on $(1, \infty)$ such that for some $t \ge 1$
	$$
	|f^{(r)}(x)| \asymp \frac{t}{x^{r+s}}
	$$
for $r=1, \dots, q+2$.
Then we have the following:
\begin{itemize}
\item[(1)]
For $\b=1$, define $\d_{q}=\frac{1}{q+2+s}$, $\D=\frac{1}{s}$, and the interval
$I_{q,s,t}:=(2^{\lfloor \log_2 t^{\d_q}\rfloor}, 2^{\lfloor \log_2 t^\D \rfloor}]$.
There exists a constant $C_{q,s}>0$ depending on $q$ and $s$ such that
\begin{equation*}
\left| \sum_{n \in I_{q,s,t}}\left(\frac{1}{n^\b}-\frac{1}{n^{2\b}}\right)\cos 2 \pi f(n)\right| \le C_{q,s}.
\end{equation*}
\item[(2)]
For $0 < \b < 1$, define $\d_{0}=\frac{1}{s+2}$, $\D=\frac{1}{s+1}$, and the interval
$I_+=(2^{\lfloor \log_2 t^{\d_0}\rfloor}, 2^{\lfloor \log_2 t^\D \rfloor}]$.
There exists a constant $C_{\b,s}>0$ depending on $\b$ and $s$ such that
\begin{equation*}
\left| \sum_{n \in I_+}\left(\frac{1}{n^\b}-\frac{1}{n^{2\b}}\right)\cos 2 \pi f(n)\right| \le C_{\b,s}t^{\frac{1-\b}{s+2}}. 
\end{equation*}
\end{itemize}
Here $\lfloor a \rfloor$ denotes the integer part of $a$.
\end{lemma}

We employ the following theorem in \cite{GK} to show the above Lemma \ref{Iq-sigma}.
This is the iterated version of \cite[Theorem 2.7, Chapter 1]{KN}.

\begin{theorem}[\cite{GK}, Theorem 2.9]\label{GK}
Let $q \ge 0$ be an integer. Suppose that $f$ has $q+2$ continuous derivatives on an interval $I \subset (N, 2N]$.
Suppose also that there is a constant $F$ such that
	\begin{equation*}
		|f^{(r)}(x)| \asymp FN^{-r}, \ x \in I,
	\end{equation*}
for $r=1, \dots, q+2$.
Then
	\begin{equation*}
		\left|\sum_{n \in I}e^{2\pi i f(n)}\right| \le c\left(F^{\frac{1}{4Q-2}}N^{1-\frac{q+2}{4Q-2}}+ F^{-1}N\right),
	\end{equation*}
where the $n$ runs over integers in $I$ in the above summation, $Q=2^q$ and $c$ is an absolute constant.
\end{theorem}

\begin{proof}[Proof of Lemma \ref{Iq-sigma}]
Let $t \ge 1$ be a constant appearing in the order of magnitude of $|f^{(r)}(x)|$.
For $0< \d_q <\D$, 
consider the interval $I_q=(2^{\lfloor \log_2 t^{\d_q}\rfloor}, 2^{\lfloor \log_2 t^{\D} \rfloor}]$.
Divide the interval $I_q$ as 
$$
I_q=\bigsqcup_{k=m_q}^{M-1}J_k
$$ 
where
	\begin{align*}
		m_q=\lfloor \log_2 t^{\d_q}\rfloor, \ &\ M=\lfloor \log_2 t^{\D} \rfloor, \\
		J_k =(2^k, 2^{k+1}], \ &\ m_q \le k \le M-1.
	\end{align*}

Applying Theorem \ref{GK} on each $I \subset J_k=(2^k, 2^{k +1}]$
with $N=2^k$,
	\begin{equation*}
		|f^{(r)}(x)|\asymp \frac{t}{N^{r+s}}, \ \ F=\frac{t}{N^s},
	\end{equation*}
for $r=1, \dots, q+2$, we have
	\begin{equation}\label{I}
		\left|\sum_{n \in I}e^{2\pi i f(n)}\right| \le c t^{\frac{1}{4Q-2}}N^{-\frac{s}{4Q-2}}N^{1-\frac{q+2}{4Q-2}}+c t^{-1}N^{1+s}.
	\end{equation}

By summation by parts on each $J_k$,
	\begin{equation*}\label{sum-parts}
		\sum_{n \in J_k}\left(\frac{1}{n^\b}-\frac{1}{n^{2\b}}\right)\cos2 \pi f(n) =\sum_{l=2^k+1}^{\infty}a_l		\sum_{n=2^k+1}^{l\wedge 2^{k+1}}\cos2 \pi f(n),
	\end{equation*}
where $a_l=\left(\frac{1}{l^\b}-\frac{1}{l^{2\b}}\right)-\left(\frac{1}{(l+1)^\b}-\frac{1}{(l+1)^{2\b}}\right)$,
and $x\wedge y=\min(x, y)$.
Taking $I=(2^k, l\wedge 2^{k+1}]\subset (N, 2N]$,
we have by (\ref{I}) that
		\begin{align*}
		\left|\sum_{n \in J_k} \left(\frac{1}{n^\b}-\frac{1}{n^{2\b}}\right)\cos 2 \pi f(n)\right|	
		&\le c_1N^{-\b}\left(t^{\frac{1}{4Q-2}}N^{1-\frac{q+2+s}{4Q-2}}+t^{-1}N^{1+s}\right)	\\
		&= c_1 t^{\frac{1}{4Q-2}}N^{1-\b-\frac{q+2+s}{4Q-2}}+c_1 t^{-1}N^{1+s-\b}	\\
		&= c_1 t^{\frac{1}{4Q-2}}2^{-k\left(-1+\b+ \frac{q+2+s}{4Q-2}\right)} + c_1 t^{-1}2^{k\left(1+s-\b\right)}.
	\end{align*}

Here, note that the exponents in the two ``$2^{\pm k}$"'s are positive:
	\begin{equation*}
		-1+\b+\frac{q+2+s}{4Q-2}=	
			\begin{cases}
				\frac{q+2+s}{4Q-2} > 0,	\ \text{when (1) $\b=1$ and $q\ge 0$}, \\
				\frac{2\b+s}{2}>0, 		\ \text{when (2) $0 < \b < 1$ and $q=0$},
			\end{cases}
	\end{equation*}
	and $1+s-\b >0$ for both cases (1) and (2).

Therefore, for every $t \ge 1$, with $2^{-m_q}\asymp t^{-\d_q}$ and $2^{M} \asymp t^\D$,
	\begin{equation}\label{eqlast}
	\begin{aligned}
		\sum_{k=m_q}^{M-1}\left|\sum_{n \in J_k} \left(\frac{1}{n^\b}-\frac{1}{n^{2\b}}\right)\cos 2\pi f(n)\right|	
		&\le c_2 t^{\frac{1}{4Q-2}}2^{-m_q\left(-1+\b+\frac{q+2+s}{4Q-2}\right)}+c_2 t^{-1}2^{M (1+s-\b)}	\\	
		&\le c_3 t^{\frac{1}{4Q-2}}t^{-\d_q \left(-1+\b+\frac{q+2+s}{4Q-2}\right)}+c_3 t^{-1}t^{\D\left(1+s-\b\right)},
	\end{aligned}
	\end{equation}
where the constant $c_3$ depends on $q$, $\b$ and $s$.
	
First, we show (1).
Let $\b=1$,
$\d_q=\frac{1}{q+2+s}$ and $\D=\frac{1}{s}$.
Then,
	\begin{equation*}
		\frac{1}{4Q-2}-\d_q\left(-1+\b+\frac{q+2+s}{4Q-2}\right)=0,
	\end{equation*}
and $-1+\D(1+s -\b)=0$.
Hence, the last sum in (\ref{eqlast}) is bounded from above by some constant $C_{q,s}$.

Next, we show (2).
Let $0 < \b < 1$, $\d_0=\frac{1}{s+2}$ and $\D=\frac{1}{s+1}$. 
When $q=0$,
it follows that
	\begin{equation*}
		\frac{1}{4Q-2}-\d_0\left(-1+\b+\frac{q+2+s}{4Q-2}\right)=\frac{1-\b}{s+2}>0,
	\end{equation*}
and $-1+\D(1+s -\b)=\frac{-\b}{s+1}< 0$.
Then we have the desired bound $C_{\b,s}t^{\frac{1-\b}{s+2}}$.
\end{proof}

\begin{proof}[Proof of Theorem \ref{fourier}]
We apply Lemma \ref{Iq-sigma} with $f(x)=\frac{t}{x^{s}}$.
First, we prove (1).
For all $\e>0$ such that $\e < \frac{1}{s}$, choose an integer $q$ satisfying that
$\d_{q}=\frac{1}{q+2+s}< \e$.
Then $\D -\d_q > \frac{1}{s}-\e>0$.
By Lemma \ref{Iq-sigma} (1),
\begin{align*}
\log |\hat \m(t)|^2 &\le \sum_{n \in I_q}\left(-\frac{2}{n}+\frac{2}{n^2}+2\left(\frac{1}{n}-\frac{1}{n^2}\right)\cos\frac{2 \pi t}{n^s}\right) \\
&\le -2(\D-\d_q)\log t + C_{\e,s},
\end{align*}
and we obtain the bound $\left(-\frac{2}{s}+2\e\right)\log t +C_{\e,s}$ for every $t \ge 1$;
hence we have
$
|\hat \m(t)| \le C_{\e,s} t^{-\frac{1}{s}+\e}.
$
For $0<s<2$ and $\e < \frac{1}{s} - \frac{1}{2}$, since $\hat \m$ is in $L^2$, the distribution has a density in $L^2$  by the Plancherel theorem \cite[Theorem 3.1, Chapter VI]{K}.
For $0<s<1$ and $\e < \frac{1}{s} - 1$, since $\hat \m$ is in $L^1$, the distribution $\m$ has a bounded and continuous density by the Fourier inversion formula.

Next, we show (2).
By Lemma \ref{Iq-sigma} (2),
\begin{align*}
\log |\hat \m(t)|^2 &\le \sum_{n \in I_+}\left(-\frac{2}{n^\b}+\frac{2}{n^{2\b}}+2\left(\frac{1}{n^\b}-\frac{1}{n^{2\b}}\right)\cos\frac{2 \pi t}{n^s}\right) \\
&\le \sum_{n \in I_+}-\frac{2}{n^\b}\left(1-\frac{1}{n^\b}\right)+C_{\b,s}t^{\frac{1-\b}{s+2}}.
\end{align*}
Since $\b<1$, the first term in the last line is bounded by
$-C_\b t^{\D(1-\b)}$.
Here $\D(1-\b)=\frac{1-\b}{s+1}$ in $(0,1)$, and this is greater than $\frac{1-\b}{s+2}$;
hence there exists $T$ such that for every $t \ge T$, we have the bound $-C_{\b,s} |t|^{\frac{1-\b}{s+1}}$.
Since $\hat \m$ decays faster than any polynomial, $\m$ has a smooth density.
\end{proof}

\begin{proof}[Proof of Corollary \ref{HausYoung}]
The decay of the Fourier transform implies that the density of $\m$ is in $C^r$ for $0 < s < 1/(r+1)$ and an integer $r \ge 0$.
We note that $\hat \m$ is in $L^p$ for every $p>s$ by Theorem \ref{fourier}(1).
For $1 \le s < 2$,
by the Hausdorff-Young inequality \cite[Theorem 3.2, Chapter VI]{K},
the density $f$ of $\m$ satisfies that
$
\|f\|_{L^q} \le \|\hat f\|_{L^p},
$
for every $s<p\le 2$, and the conjugate $q$ of $p$.
Therefore
the density of $\m$ is in $L^q$ for every $2 \le q < s/(s-1)$.
The density of $\m$ is a priori in $L^1$; by the inequality $\|f\|_{L^r}\le \|f\|_{L^p}^\th \|f\|_{L^q}^{1-\th}$ for $1\le p<r<q \le \infty$, where $\th=\left(\frac{1}{r}-\frac{1}{q}\right)/\left(\frac{1}{p}-\frac{1}{q}\right)$,
we see the density of $\m$ is in $L^q$ for every $1 \le q < s/(s-1)$ as well.
\end{proof}

Theorem \ref{Sinfty} follows from the same argument as Theorem \ref{fourier} and Corollary \ref{HausYoung}.

\begin{proof}[Proof of Theorem \ref{Sinfty}]
Fix $s=1$ and apply to Lemma \ref{Iq-sigma} (1) for $f(x)=-t\log\left(1-\frac{1}{x}\right)$.
Note that on $[2, \infty)$, we have $|f^{(r)}| \asymp \frac{t}{x^{r+1}}$ for $0 \le r \le q+2$.
It follows that the Fourier transform of the distribution $\m_\infty$ of $S_\infty$ satisfies the same estimate as the one for $\m$ in the case when $\b=1$ and $s=1$ in Theorem \ref{fourier}(1).
Then, as in Corollary \ref{HausYoung}, the distribution $\m_\infty$ has a density in $L^q$ for all $1 \le q < \infty$.
Since $V_\infty=e^{-S_\infty}$, the distribution of $V_\infty$ is absolutely continuous with respect to Lebesgue measure.
\end{proof}

\section{Completion of the proof of Theorem \ref{main}.}\label{ac}

We consider the case when $s>0$ and $\b=1$, namely,
$S:=\sum_{n=1}^{\infty}I_n/n^s$, where the $I_n$ are independent and $I_n=1$ with probability $1/n$ and $I_n = 0$ otherwise.

\begin{theorem}\label{all-s}
For every $s>0$ and $\b=1$, the distribution $\m$ of $S$ is absolutely continuous with respect to Lebesgue measure.
\end{theorem}

\begin{proof}
Fix $s>0$.
It suffices to show that there is a sequence $(A_m)$ of events with $\Pb(A_m) \to 1$ as $m \to \infty$, such that conditioned on $A_m$, the distribution of $S$ is absolutely continuous.
Indeed, this follows from the formula,
$$
\Pb(S \in B | A)=\frac{\Pb(\{S \in B\} \cap A)}{\Pb(A)} \ge \frac{\Pb(S \in B)-(1-\Pb(A))}{\Pb(A)},
$$
since if there is a set $B$ of Lebesgue measure $0$ with $\Pb(S \in B)>0$ then also
$\Pb(S \in B | A)>0$ when $\Pb(A)$ is sufficiently close to $1$.

Rewrite the series as
$$
S=1+\sum_{k=0}^{\infty}Y_k,
$$
where $Y_k$ is the portion of the sum taken on the $k$-th block $(2^k, 2^{k+1}]$, i.e.,
$$
Y_k=\sum_{n=2^k+1}^{2^{k+1}}\frac{I_n}{n^s}.
$$
Define a sequence of independent Bernoulli random variables $(M_k)$ by setting $M_k=1$ if and only if $I_n=1$ for exactly one $n$ in the $k$-th block.
It is straightforward to check that there exists some $p>0$ such that
$
\Pb(M_k=1)\ge p
$
for all $k$ (e.g.\ $p=\frac{1}{2}\log 2$).
Thus the sequence $(M_k)$ dominates an i.i.d.\ sequence of Bernoulli random variables with parameter $p$.
We proceed to define the events $A_m$.
Define a large constant $a$ by
$
a=2/\log\left(1/(1-p)\right).
$
Define a sequence of events $(B_m)$ by
$$
B_m=\bigcup_{m \le k \le m+\lfloor a\log m \rfloor}\{M_k=1\}.
$$
Define
$
A_m=\bigcap_{l \ge m}B_l.
$
To check that $\Pb(A_m) \to 1$ as $m \to \infty$, we need only notice that the choice of $a$ ensures that
$
\Pb(B_l^{\mathsf{c}}) \le 1/l^2
$
and we have
$
\Pb(A_m^{\mathsf{c}}) \le \sum_{l=m}^{\infty}\Pb(B_l^{\mathsf{c}}),
$
where $A^{\mathsf{c}}$ denotes the complement of $A$.

From now on fix $m$ and condition on the entire sequence $(M_k)$, at a sample point where the event $A_m$ is satisfied.
Observe that the random variables $(Y_k)$ are still independent under this conditioning.
For a $k$ such that $M_k=1$ the remaining randomness in $Y_k$ is exactly which $n$ is the single $n$ in the $k$-th block for which $I_n=1$.
This $n$ is distributed in the block according to the probabilities $1/(z_k(n-1))$, where $z_k$ is a normalising constant which tends to $\log 2$ as $k$ grows, and $z_k \ge \log 2$.

Write $\hat \m_M(t)$ for the Fourier transform of the distribution of $S$ conditioned on a sample sequence $M=(M_k)$, 
and $\hat \m_k(t)$ for the Fourier transform of the distribution of $Y_k$ conditioned on $M_k$.
Here we define $\hat \m_M(t)$ as the Fourier transform of the regular conditional probability given $M=(M_k)$.
(See e.g., \cite[Chapter 4.3]{B} on the existence of a regular conditional probability.)
The Fourier transform of the distribution of $S$ conditioned on $A_m$ is given by
\begin{equation}\label{Fcond}
\frac{1}{\Pb(A_m)}\int_{A_m}\hat \m_{M(\o)}(t)\Pb(d\o).
\end{equation}
We have
$
\hat \m_M(t)=e^{2\pi i t}\prod_{k=0}^{\infty}\hat \m_k(t)
$
for $\Pb$-almost every sequence $M=(M_k)$,
and $|\hat \m_k(t)|\le 1$ for all $k$ and $t$.
If $k$ is such that $M_k=1$, then we have
$$
\hat \m_k(t)=\sum_{n=2^k+1}^{2^{k+1}}\frac{1}{z_k(n-1)}\exp{\frac{2 \pi i t}{n^s}}.
$$
Apply Theorem \ref{GK} in the case where
$q=0$, $N=2^k$,
$f(x)=t/x^s$ and $F=t/N^s$, then for $I \subset (N, 2N]$,
we have
\begin{equation*}
\left|\sum_{n \in I}e^{2 \pi if(n)}\right| \le c\left(t^{\frac{1}{2}}N^{-\frac{s}{2}}+t^{-1}N^{1+s}\right),
\end{equation*}
where the constant $c$ is absolute.
By summation by parts on $(N, 2N]$ as in the proof of Lemma \ref{Iq-sigma},
using $z_k\ge \log2$,
we have
\begin{align*}
\left|\hat \m_k(t)\right|=\left|\sum_{n=N+1}^{2N}\frac{1}{z_k(n-1)}\exp{\frac{2 \pi i t}{n^s}}\right|
&\le cN^{-1}\left(t^{\frac{1}{2}}N^{-\frac{s}{2}}+t^{-1}N^{1+s}\right) \\
&\le c t^{\frac{1}{2}}2^{-k\left(\frac{s}{2}+1\right)}+ct^{-1}2^{ks}.
\end{align*}
Fix some $\s_1$ and $\s_2$ such that $s<\s_1<\s_2<s+2$ (e.g., $\s_1=s+0.1$ and $\s_2=s+1.9$).
Then for $k$ such that $\frac{1}{\s_2}\log_2 t<k<\frac{1}{\s_1}\log_2 t$, we have
\begin{equation*}
\frac{1}{2}\log_2 t-k\left(\frac{s}{2}+1\right) < -\frac{s+2-\s_2}{2\s_2}\log_2 t=-\d_2 \log_2 t
\end{equation*}
and
\begin{equation*}
-\log_2 t+ks< -\left(1-\frac{s}{\s_1}\right)\log_2 t=-\d_1\log_2 t,
\end{equation*}
where $\d_1$ and $\d_2$ are positive and depend only on $s$, $\s_1$ and $\s_2$.
Define $\d:=\min(\d_1, \d_2)>0$.
To summarize,
if $k$ is such that $M_k=1$ and $\frac{1}{\s_2}\log_2 t<k<\frac{1}{\s_1}\log_2 t$,
then we have
\begin{equation}\label{poly}
\left|\hat \m_k(t)\right| \le ct^{-\d}
\end{equation}
for every $t\ge 1$, where the constant $c$ is absolute.

By the definition of $A_m$, the number of $k$ in the interval
$
\left[\frac{1}{\s_2}\log_2 t, \frac{1}{\s_1}\log_2 t\right]
$
having $M_k=1$ tends to infinity as $t$ tends to infinity.
Since (\ref{poly}) holds for all $k$ in this interval having $M_k=1$ and since there are more and more of these as $t$ grows,
we conclude that $\hat \m_M(t)$, as a product of all $\hat \m_k(t)$, decays faster than 
$
(c|t|^{-\d})^L
$
in $t$ for all $L>0$.
This decay of $\hat \m_M(t)$ holds uniformly on $A_m$ modulo $\Pb$-measure null set.
Therefore (\ref{Fcond}) decays faster than polynomially in $t$.
This proves that conditioned on $A_m$, the distribution of $S$ is smooth, in particular, absolutely continuous,
and concludes the proof.
\end{proof}

\begin{theorem}\label{unbounded}
For every $s>1$ and $\b=1$,
the density of $\m$ is unbounded on every interval in its support.
\end{theorem}

\begin{proof}
Let $D_N=\{\sum_{n=1}^N \e_n/n^s \ | \ \text{$\e_n=0$ or $1$}\}$ be the set of all possible values of the sums up to the $N$-th term in the series $S$.
Note that $\bigcup_N D_N$ is dense in the support of $\m$.
For every $M>N$, we write the series as the sum of three independent random variables: $S=S_1+S_2+S_3$, where
$$
S_1:=\sum_{n=1}^N I_n/n^s, \ S_2:=\sum_{n=N+1}^M I_n/n^s, \text{ and } S_3:=\sum_{n=M+1}^\infty I_n/n^s.
$$
We have $\Pb\left(S_1=x\right)=c_x>0$ for each $x \in D_N$, and
$\Pb\left(S_2=0\right)=\prod_{n=N+1}^{M}\left(1-1/n\right)\ge C/M$.
Since $\E S_3 \le C/M^s$, we have
$\Pb\left(S_3 \le 2C/M^s\right) \ge 1/2$ by Markov's inequality.
Therefore we obtain
$\Pb\left(S \in [x, x+ 2C/M^s]\right) \ge c'_x/M$ for some constant $c'_x>0$ depending only on $x$ and for arbitrary large $M$.
Thus, for every $s>1$, the density is unbounded on every interval in its support.
\end{proof}

\begin{remark}\label{sharp}
We see the sharpness of the estimate of Fourier transform $\hat \m(t)$ in part 1 of Theorem \ref{main} for $1<s<2$.
Since $\m$ has a density in $L^q$ for all $2 \le q < s/(s-1)$ by Corollary \ref{HausYoung},
H\"older's inequality gives that
$
\Pb\left(S \in [x, x+\e]\right)=o(\e^{1/p}),
$
as $\e \to 0$,
for every $s<p \le 2$.
This is sharp.
Indeed, in the above argument,
for each $x$ in the union of the $D_N$, and for $\e=2C/M^s$,
we have
$
\Pb\left(S \in [x, x+\e]\right) \ge c_x\e^{1/s}.
$

\end{remark}

\section{Singularity for the prime numbers sequence.}\label{primes}

\begin{proof}[Proof of Theorem \ref{sing}]
Let $S:=S_{primes}$.
Fix $0<\e<s$.
Decompose the series into three parts:
$S=S_1+S_2+S_3$, where
$$
S_1:=\sum_{p \le N^\e}\frac{I_p}{p^s}, \
S_2:=\sum_{N^\e < p < N}\frac{I_p}{p^s}, \ \text{and} \
S_3:=\sum_{N \le p}\frac{I_p}{p^s}.
$$
Since
$\E S_3=\sum_{N\le p}1/p^{1+s} \le C/N^s$ for some constant $C$,
we have
$\Pb(S_3\ge 2C/N^s)\le 1/2$ by Markov's inequality.
Hence $\Pb(S_3\le 2C/N^s)\ge 1/2$.

Let us show that the probability that $S_2=0$ is bounded away from below by a positive constant independent of $N$.
We use Mertens' theorem,
\begin{equation}\label{Mertens}
\prod_{p <N}\left(1-\frac{1}{p}\right)=\frac{e^{-\g+o(1)}}{\log N},
\end{equation}
where $\g$ is the Euler-Mascheroni constant (\cite[Theorem 429]{HW} and \cite{M}).
Then we have
$$
\Pb(S_2=0)=\prod_{N^\e<p<N}\left(1-\frac{1}{p}\right)=\e e^{o(1)} \ge \e c_1,
$$
where $c_1$ is absolute.

Fix a positive integer $m< N^\e$.
Define the sequence $\{a_p(m)\}_{p<N^\e}$ such that
$a_p(m)$ is $1$ if $p|m$, and $0$ otherwise.
Let $x_m:=\sum_{p<N^\e}a_p(m)/p^s$.
We claim that the probability that $S_1$ coincides with some $x_m$ is bounded away from $0$ by a constant independent of $N$.
We consider the event $A_m:=\{I_p=a_p(m) \ \text{for all $p<N^\e$}\}$.
Notice that the $A_m$ are not disjoint (e.g., $A_2=A_4$), but $A_m$ for square-free $m$'s are disjoint.
Then,
\begin{align*}
\Pb\left(S_1=x_m \ \text{for some $m < N^\e$}\right)
	&\ge \Pb\left(\bigcup_{m<N^\e}A_m\right) \\
	&=\sum_{m<N^\e, \text{$m$ is square-free}}\Pb(A_m) \\
	&=\sum_{m<N^\e}|\m(m)|\prod_{p|m}\frac{1}{p-1}\prod_{p<N^\e}\left(1-\frac{1}{p}\right),
\end{align*}
where $\m(m)$ is the M\"obius function, i.e.,
\begin{itemize}
\item[(i)] $\m(1)=1$,
\item[(ii)] $\m(m)=0$ if $m$ has a squared factor,
\item[(iii)] $\m(p_1p_2\cdots p_k)=(-1)^k$ if all the primes $p_1, p_2, \dots, p_k$ are different.
\end{itemize}
The last sum is bounded from below by
\begin{equation}\label{lb}
\sum_{m<N^\e}\frac{|\m(m)|}{m}\prod_{p<N^\e}\left(1-\frac{1}{p}\right)
\end{equation}
since $\prod_{p|m}p \le m$.
Note that the number of square-free numbers up to $x$ grows linearly in $x$; more precisely,
\begin{equation*}\label{square-free}
\sum_{m \le x}|\m(m)|=\frac{6}{\pi^2}x +O\left(\sqrt{x}\right)
\end{equation*}
by \cite[Theorem 334]{HW}.
By summation by parts, we have that
\begin{equation*}
\sum_{m \le N^\e}\frac{|\m(m)|}{m}=\frac{6}{\pi^2}(1+o(1))\log N^\e.
\end{equation*}

This and Mertens' theorem (\ref{Mertens}) imply that the sum (\ref{lb}) is at least some positive constant $c_2$ which is independent of $N$ and $\e$.

Let us define the set $B_N$ as a union of intervals
$\left[x_m, x_m+2C/N^s\right]$ for $m < N^\e$.
Since $S_1$, $S_2$ and $S_3$ are independent,
we have $\Pb(S \in B_N) \ge c_1c_2\e/2$.
On the other hand, the Lebesgue measure of $B_N$ is at most $2CN^{\e-s}$, which tends to $0$ as $N \to \infty$.
Hence the distribution cannot be absolutely continuous.
Recall that the distribution is continuous if and only if 
$$
\sum_{p}\left(1-\sup_{x \in \R}\Pb\left(\frac{I_p}{p^s}=x\right)\right)=\infty
$$
\cite[Lemma 1.22]{Ell}.
It is satisfied since $\sum_p 1/p=\infty$ (e.g., \cite[Theorem 427]{HW}).
The Jessen-Wintner law of pure type \cite{JW} implies that the distribution is non-atomic singular.
\end{proof}

\begin{remark}\label{gen}
By a straight forward adaptation to the above proof,
Theorem \ref{sing} is generalized as follows:
Let $(a_p)_{p; primes}$ be a sequence of real numbers such that
\begin{itemize}
\item[(i)] $\sum_{p, a_p\neq 0}\frac{1}{p}=\infty$,
\item[(ii)] $\sum_{p}\frac{|a_p|}{p}<\infty$,
\item[(iii)] $\sum_{p>x}\frac{|a_p|}{p}=o\left(x^{-c}\right)$ for some $c>0$,
\end{itemize}
and $(I_p)$ be an independent sequence with value $0$ or $1$ with probability $1-1/p$ or $1/p$, respectively.
Then $S=\sum_p a_pI_p$ converges almost surely, and its distribution is non-atomic singular.
\end{remark}

\section{Further questions}\label{questions}

Below we list some natural questions about the critical case when $s=1$ and $\b=1$.
Let $\b=1$.
By part 1 of Theorem \ref{main}, for $0 < s < 1$ the distribution $\m$ has a bounded continuous density,
while for $s>1$ it has an unbounded density.
In the case when $s=1$,
the density of $\m$ is in $L^q$ for every $1 \le q<\infty$ by Corollary \ref{HausYoung}.
In fact, we conjecture the following.

\begin{conjecture}
For $s =1$ and $\b=1$, the density is bounded.
\end{conjecture}

We also ask the following question about the critical case.

\begin{question}
For $s=1$ and $\b=1$, is the density discontinuous?
\end{question}

\textbf{Acknowledgements.}
We thank Hiroko Kondo for making beautiful pictures, and an anonymous referee for reading the paper carefully, for correction of the proof of Theorem \ref{sing}, for Remark \ref{gen} and for additional references.
R.T.\ thanks the Theory Group of Microsoft Research where this work was initiated for its kind hospitality, and Kouji Yano for helpful discussions.

\end{document}